\documentclass[letterpaper,10pt,conference]{ieeeconf}
\IEEEoverridecommandlockouts
\makeatletter

\let\proof\@undefined
\let\endproof\@undefined
\makeatother

\usepackage{graphicx,psfrag}
\usepackage[noend]{algorithmic}
\usepackage{xspace}
\usepackage{amssymb,amsmath,amscd,mathrsfs}
\usepackage{amsthm}
\usepackage{comment,color}
\usepackage{subfigure}
\usepackage{url}
\usepackage{multirow}
\usepackage{etoolbox}

\usepackage[bookmarks=true]{hyperref}
\usepackage{fixltx2e}

\usepackage{color}
\usepackage{amsmath}
\usepackage{amssymb}
\usepackage{graphicx}
\usepackage{comment,xspace}
\usepackage{fancybox}

\newcommand{\real}{{\mathbb{R}}}
\newcommand{\reals}{\real}

\newtheorem{theorem}{Theorem}[section]
\newtheorem{proposition}[theorem]{Proposition}
\newtheorem{lemma}[theorem]{Lemma}
\newtheorem{remark}[theorem]{Remark}

\newtheorem{definition}[theorem]{Definition}

\newcommand{\argmin}{\operatornamewithlimits{argmin}}

\newcommand{\expectation}[1]{\mbox{$\mathbb{E}\left[#1\right]$}}

\newcommand{\fil}{\mathcal F}
\newcommand{\cs}{\mathcal Z}
\newcommand{\csd}{\mathcal U}

\newcommand{\risk}{\rho}

\renewcommand{\baselinestretch}{0.91}

\newtoggle{paper}

\togglefalse{paper}

\title{\LARGE \bf
A Uniform-grid Discretization Algorithm for \\Stochastic Control with Risk Constraints
}

\author{Yin-Lam Chow, Marco Pavone
\thanks{Y.-L. Chow and M. Pavone are with the Department of Aeronautics and Astronautics, Stanford University, Stanford, CA 94305, USA. Email: {\tt \{ychow, pavone\}@stanford.edu}.}}

\begin{document}

\maketitle
\thispagestyle{empty}
\pagestyle{empty}

\begin{abstract}
In this paper, we present a discretization algorithm for finite horizon risk constrained dynamic programming algorithm in \cite{Chow_Pavone_13_1}. Although in a theoretical standpoint,  Bellman's recursion provides a systematic way to find optimal value functions and generate optimal history dependent policies, there is a serious computational issue. Even if the state space and action space of this constrained  stochastic optimal control problem are finite, the spaces of risk threshold and the feasible risk update are closed bounded subset of real numbers. This prohibits any direct applications of unconstrained finite state iterative methods in dynamic programming found in \cite{bertsekas_05}. In order to approximate Bellman's operator derived in \cite{Chow_Pavone_13_1}, we discretize the continuous action spaces and formulate a finite space approximation for the exact dynamic programming algorithm. We will also prove that the approximation error bound of optimal value functions is bound linearly by the step size of discretization. Finally, details for implementations and possible modifications are discussed.
\end{abstract}
\section{Introduction}

Constrained stochastic optimal control problems naturally arise in decision-making problems where one has to consider multiple objectives. Instead of introducing an aggregate utility function that has to be optimized, one consider a setup where one cost function is to be minimized while keeping the other cost functions below some given bounds. Application domains are broad and include engineering, finance, and logistics. Within a constrained framework, the most common setup is, arguably, the optimization of a \emph{risk-neutral expectation} criterion subject to a \emph{risk-neutral} constraint \cite{Chen_07, Chow_Pavone_13_1}. This model, however, is not suitable in scenarios where risk-aversion is a key feature of the problem setup. To introduce risk aversion, in \cite{Chow_Pavone_13_1} the authors studied stochastic optimal control problems with risk constraints, where risk is modeled according to \emph{dynamic, time-consistent risk metrics} \cite{rus_09, Shapiro_Dentcheva_Ruszczynski_09}. These metrics have the desirable property of ensuring rational consistency of risk preferences across multiple periods \cite{Shapiro_Dentcheva_Ruszczynski_09}. (In contrast, traditional static risk metrics, such as conditional value at risk, can lead to potentially ``inconsistent" behaviors, see \cite{Iancu_11} and references therein.) In particular, in \cite{Chow_Pavone_13_1}, the authors developed a dynamic programming approach that allows to (formally) compute the optimal costs  by value iteration via a \emph{constrained} dynamic programming operator. The key idea is that due to the compositional structure of dynamic risk constraints, the optimization problem can be cast as a Markov decision problem (MDP) on an augmented state space where Markov policies are optimal (as opposed to the original problem) and Bellman's recursion can be applied. Henceforth, we will refer to such augmented MDP as AMDP. However, even if both the state space and action spaces for the original optimization problem are assumed to be finite, the augmented state in AMDP contains state variables that are \emph{continuous} and lie in bounded subsets of the real numbers. Hence, apart from a few cases when an analytical solution is available, the problem must be solved numerically.

Accordingly, the objective of this paper is to develop a numerical method for the solution of stochastic optimal control problems with dynamic, time-consistent risk measures. The approach is to discretize the continuous states in AMDP. Numerical algorithms for the solution of continuous MDPs is indeed a fairly mature field. In \cite{whitt_78, gordon_99, Chow_Tsitsiklis_91}, multi-grid state/action space discretization methods are developed with bounds available on how fine the discretization should be in order to achieve a desired accuracy. In \cite{Rust_97}, the grid for discretization  is chosen via randomized sampling techniques and Monte Carlo methods. In  \cite{tsitsiklis_van_roy_96}, the value functions are approximated by a finite number of basis functions. Variable resolution grid sampling techniques have been  proposed in \cite{Munos_Moore_00, Munos_Moore_98, Munos_Moore_97}. However, in general, these results assume that the dynamic programming operator is \emph{unconstrained}, i.e., actions and future states are only constrained to lie in their respective feasible sets. In contrast, the dynamic programming operator for AMDP \emph{constrains} actions and future states in a more complicated fashion (see Section \ref{sec:prelim} for more details). This precludes the application of current approximation algorithms to the numerical solution of AMDP.

Our approach is to extend the uniform grid discretization approximation developed in \cite{Chow_Tsitsiklis_91}. This requires the development of novel Lipschitz bounds for constrained dynamic programming operators. We show that convergence is linear in the step size, which is the same convergence rate for discretization algorithms for unconstrained dynamic programming operators \cite{Chow_Tsitsiklis_91}. The importance of our result is fourfold. First, we provide a sound numerical method for the solution of AMDP. Second, our results provide the basis to develop more sophisticated approximation algorithms (e.g., variable grid size, reinforcement learning, etc.) for the solution of stochastic optimal control problems with dynamic, time-consistent risk constraints.  Third, a particular type or dynamic, time-consistent ``risk" constraint is, of course, the risk neutral expectation. Hence, our results provide as a particular case a numerical algorithm to solve the dynamic programing equations that arise in traditional constrained stochastic optimal control problems \cite{Chen_07}. To the best of our knowledge, this is the first practical algorithm to solve such dynamic programming equations. Finally, the ideas and techniques introduced in the current paper could be useful for the development of approximation algorithms for other types of constrained dynamic programming operators.

The rest of the paper is structured as follows. In Section \ref{sec:prelim} we present background material for this paper, in particular about dynamic, time-consistent risk metrics and stochastic optimal control with dynamic risk constraints \cite{Chow_Pavone_13_1}. In Section \ref{sec:dis} we present and theoretically study a uniform grid approximation algorithm for the augmented MDP; in particular, we show that the error bound is linear to the discretization step size. In Section \ref{sec:num}, we study by numerical simulations the performance of the proposed algorithm and discuss details of implementations using Branch and Bound techniques. Finally, in Section \ref{sec:conc}, we draw our conclusions and offer directions for future work.

\section{Preliminaries}\label{sec:prelim}
In this section we provide some background for the theory of dynamic, time-consistent risk metrics and stochastic optimal control with dynamic risk constraints, on which we will rely extensively later in the paper.
\subsection{Notations}
In this paper, given a real-valued function $f$, $\text{dom}(f)$ denotes its domain and $\text{epi} f$ denotes its epigraph (i.e., the set of points lying on or above its graph). Let $\nu$ and $\mu$ be two probability measures on the same measurable space, then $\nu\ll\mu$ denotes that $\nu$ is absolutely continuous with respect to $\mu$ (i.e., $\nu(E)=0$ for every set $E$ for which $\mu(E) =0$).

\subsection{Markov Decision Processes}
A finite Markov Decision Process (MDP) is a four-tuple $(S, U, Q, U(\cdot))$, where $S$, the state space, is a finite set; $U$, the control space, is a finite set; for every $x\in S$, $U(x)\subseteq U$ is a nonempty set which represents the set of admissible controls when the system state is $x$; and, finally, $Q(\cdot|x,u)$ (the transition probability) is a conditional probability on $S$ given the set of admissible state-control pairs, i.e., the sets of pairs $(x,u)$ where $x\in S$ and $u\in U(x)$.

Define the space $H_k$ of admissible histories up to time $k$ by $H_k = H_{k-1} \times S\times U$, for $k\geq 1$, and $H_0=S$. A generic element $h_{0,k}\in H_k$ is of the form $h_{0,k} = (x_0, u_0, \ldots , x_{k-1}, u_{k-1}, x_k)$. Let $\Pi$ be the set of all deterministic policies with the property that at each time $k$ the control is a function of $h_{0,k}$. In other words, $\Pi := \Bigl \{ \{\pi_0: H_0 \rightarrow U,\, \pi_1: H_1 \rightarrow U, \ldots\} | \pi_k(h_{0,k}) \in U(x_k) \text{ for all } h_{0,k}\in H_k, \, k\geq 0 \Bigr\}$.

\subsection{Dynamic, time-consistent, risk measures}
Consider a probability space $(\Omega, \fil, P)$, a filtration $\fil_1\subset \fil_2 \cdots \subset \fil_N \subset \fil$, and an adapted sequence of random variables $Z_k$, $k\in \{0, \cdots,N\}$. We assume that $\fil_0 = \{\Omega, \emptyset\}$, i.e., $Z_0$ is deterministic. In this paper we interpret the variables $Z_k$ as stage-wise costs. For each $k\in\{1, \cdots, N\}$, define the spaces of random variables with finite $p$th order moment as $\cs_k:=  L_p(\Omega, \fil_k, P)$, $p\in [1,\infty]$; also, let $\cs_{k, N}:=\cs_k \times \cdots \times \cs_N$.

Roughly speaking, a dynamic risk measure is said time consistent if it is such that when a $Z$ cost sequence is deemed less risky than a $W$ cost sequence from the perspective of a future time $k$, and both sequences yield identical costs from the current time $l$ to the future time $k$, then the $Z$ sequence is deemed as less risky at the current time $l$. It turns out that dynamic, time-consistent risk metrics can be constructed by ``compounding" one-step conditional risk measures, which are defined as follows.


\begin{definition}[Coherent one-step conditional risk measures]
A coherent one-step conditional risk measures is a mapping $\risk_k:\cs_{k+1}\rightarrow \cs_k$, $k\in\{0,\ldots,N\}$, with the following four properties:
\begin{itemize}
\item Convexity: $\risk_k(\lambda Z + (1-\lambda)W)\leq \lambda\risk_k(Z) + (1-\lambda)\risk_k(W)$, $\forall \lambda\in[0,1]$ and $Z,W \in\cs_{k+1}$;
\item Monotonicity:  if $Z\leq W$ then $\risk_k(Z)\leq\risk_k(W)$, $\forall Z,W \in\cs_{k+1}$;
\item Translation invariance:  $\risk_k(Z+W)=Z + \risk_k(W)$, $\forall Z\in\cs_k$ and $W \in \cs_{k+1}$;
\item Positive homogeneity: $\risk_k(\lambda Z) = \lambda \risk_k(Z)$, $\forall Z \in \cs_{k+1}$ and $\lambda\geq 0$.
\end{itemize}
\end{definition}

Then, the following results characterize dynamic, time-consistent risk metrics \cite{rus_09}.
\begin{theorem}[Dynamic, time-consistent risk measures]\label{thrm:tcc}
Consider, for each $k\in\{0,\cdots,N\}$, the mappings $\risk_{k,N}:\cs_{k, N}\rightarrow\cs_k$ defined as
\begin{equation}\label{eq:tcrisk}
\begin{split}
\risk_{k,N} &= Z_k + \risk_k(Z_{k+1} + \risk_{k+1}(Z_{k+2}+\ldots+\\
    &\qquad\risk_{N-2}(Z_{N-1}+\risk_{N-1}(Z_N))\ldots)),
\end{split}
\end{equation}
where the $\risk_k$'s are coherent one-step risk measures. Then, the ensemble of such mappings is a time-consistent dynamic risk measure.
\end{theorem}

In this paper we consider a (slight) refinement of the concept of dynamic, time-consistent risk measure, which involves the addition of a Markovian structure \cite{rus_09}.
\begin{definition}[Markov dynamic risk measures]\label{def:Markov}
Let $\mathcal V:=L_p(S, \mathcal B, P)$ be the space of random variables on $S$ with finite $p$\emph{th} moment. Given a controlled Markov process $\{x_k\}$, a Markov dynamic risk measure is a dynamic, time-consistent risk measure if each coherent one-step risk measure $\risk_k:\cs_{k+1}\rightarrow \cs_k$ in equation \eqref{eq:tcrisk} can be written as:
\begin{equation}\label{eq:Markov}
\risk_k(V(x_{k+1})) = \sigma_k(V(x_{k+1}),x_k, Q(x_{k+1} |x_k, u_k)),
\end{equation}
for all $V(x_{k+1})\in \mathcal V$ and $u\in U(x_k)$, where $\sigma_k$ is a coherent one-step risk measure on $\mathcal V$ (with the additional technical property that for every $V(x_{k+1})\in \mathcal V$ and $u\in U(x_k)$ the function $x_k \mapsto \sigma_k(V(x_{k+1}), x_k, Q(x_{k+1}|x_k, u_k))$ is an element of $\mathcal V$).
\end{definition}
In other words, in a Markov dynamic risk measures, the evaluation of risk is not allowed to depend on the whole past.

\subsection{Stochastic optimal control with dynamic, time-consistent risk constraints}

Consider an MDP and let $c : S \times U \rightarrow \reals$ and $d : S \times U \rightarrow \reals$ be functions which denote costs associated with state-action pairs. Given a policy $\pi\in \Pi$, an initial state $x_0\in S$, and an horizon $N\geq 1$, the cost function is defined as
\[
J^{\pi}_N(x_0):=\expectation{\sum_{k=0}^{N-1}\, c(x_k, u_k)},
\]
and the risk constraint is defined as
\[
R^{\pi}_N(x_0):= \risk_{0,N}\Bigl(d(x_0,u_0), \ldots, d(x_{N-1},u_{N-1}),0\Bigr),
\]
where $\risk_{k,N}(\cdot)$, $k\in \{0,\ldots, N-1\}$, is a Markov dynamic risk measure (for simplicity, we do not consider terminal costs, even though their inclusion is straightforward).  The problem is then as follows:
\begin{quote} {\bf Optimization problem $\mathcal{OPT}$} --- Given an initial state $x_0\in S$, a time horizon $N\geq 1$, and a risk threshold $r_0 \in \reals$, solve
\begin{alignat*}{2}
\min_{\pi \in \Pi} & & \quad&J^{\pi}_N(x_0) \\
\text{subject to} & & \quad&R^{\pi}_N(x_0) \leq r_0.
\end{alignat*}
\end{quote}
If problem $\mathcal OPT$ is not feasible, we say that its value is $\infty$. In \cite{Chow_Pavone_13_1} the authors developed a dynamic programing approach to solve this problem. To define the value functions, one needs to define the tail subproblems. For a given $k\in \{0,\ldots,N-1\}$ and a given state $x_k\in S$, we define the \emph{sub-histories} as $h_{k,j}:=(x_k, u_k, \ldots,x_j)$ for $j\in \{k,\ldots, N\}$; also, we define the \emph{space of truncated policies} as $\Pi_k:=\Bigl \{ \{\pi_k, \pi_{k+1}, \ldots\} | \pi_j(h_{k,j}) \in U(x_j) \text{ for } j\geq k \Bigr\}$. For a given stage $k$ and state $x_k$, the cost of the tail process associated with a policy $\pi\in \Pi_k$ is simply $J^{\pi}_N(x_k):=\expectation{\sum_{j=k}^{N-1}\, c(x_j, u_j)}$. The risk associated with the tail process is:
\[
R^{\pi}_N(x_k):= \risk_{k,N}\Bigl(d(x_k,u_k), \ldots, d(x_{N-1},u_{N-1}),0\Bigr),
\]
The tail subproblems are then defined as
\begin{alignat}{2}
\min_{\pi \in \Pi_k} & & \quad&J^{\pi}_N(x_k) \label{problem_SOCP}\\
\text{subject to} & & \quad&R^{\pi}_N(x_k) \leq r_k(x_k),\label{constraint_SOCP}
\end{alignat}
for a given (undetermined) threshold value $r_k(x_k) \in \reals$ (i.e., the tail subproblems are specified up to a threshold value).

For each $k\in\{0,\ldots, N-1\}$ and $x_k\in S$, we define the set of feasible constraint thresholds as
\[
\Phi_k(x_k):=[\underline{R}_N(x_k), \overline{R}_{N,k}],\quad \Phi_N(x_N):=\{0\},
\]
where  $\underline{R}_N(x_k):=\min_{\pi\in \Pi_k} \, R_N^{\pi}(x_k)$,
and $\overline{R}_{N,k}= (N-k)\risk_{\text{max}}$. The value functions are then defined as follows:
\begin{itemize}
\item If $k<N$ and $r_k \in \Phi_k(x_k)$:
\begin{alignat*}{2}
V_k(x_k, r_k)  = &\min_{\pi \in \Pi_k} & \quad&J^{\pi}_N(x_k) \\
&\text{subject to} & &R^{\pi}_N(x_k) \leq r_k.
\end{alignat*}
\item iI $k\leq N$ and $r_k \notin \Phi_k(x_k)$:
\[
V_k(x_k, r_k)  = \infty;
\]
\item when $k=N$ and $r_N=0$:
\[
V_N(x_N,r_N) = 0.
\]
\end{itemize}

Let $B(S)$ denote the space of real-valued bounded functions on $S$, and $B(S \times \reals)$ denote the space of real-valued bounded functions on $S\times \reals$. For $k\in \{0, \ldots, N-1\}$, we define the dynamic programming operator $T_k[V_{k+1}] : B(S \times \reals) \mapsto B(S \times \reals)$ according to the equation:

\begin{equation}\label{eq:T}
\begin{split}
T_k[V_{k+1}]&(x_k, r_k) := \inf_{(u,r^{\prime})\in F_k(x_k, r_k)} \, \biggl\{c(x_k,u) \, \,+\\
& \ \sum_{x_{k+1}  \in S} \, Q(x_{k+1}|x_k,u)\, V_{k+1}(x_{k+1}, r^{\prime}(x_{k+1})) \biggr\},
\end{split}
\end{equation}
where $F_k$ is the set of control/threshold \emph{functions}:
\begin{equation*}
\begin{split}
F_k(x_k,& r_k):= \biggr\{(u, r^{\prime}) \Big | u\in U(x_k), r^{\prime}(x^{\prime}) \in \Phi_{k+1}(x^{\prime}) \text{ for}\\& \text{all } x^{\prime} \in S, \text{ and } d(x_k, u) + \risk_k(r^{\prime}(x_{k+1}))  \leq r_k\biggl\}.
\end{split}
\end{equation*}
If $F_k(x_k,r_k) = \emptyset$, then $T_k[V_{k+1}](x_k, r_k)=\infty$.

Note that, for a given state and threshold constraint, set $F_k$ characterizes the set of feasible pairs of actions and subsequent constraint
thresholds.

\begin{theorem}[Bellman's equation with risk constraints]\label{TC_good}
For all $k\in \{0, \ldots, N-1\}$ the optimal cost functions satisfy the Bellman's equation:
\[
V_k(x_k, r_k) = T_k[V_{k+1}](x_k, r_k).
\]
\end{theorem}

\subsection{Representation theorems}
A key result that will be heavily exploited in this paper is the following representation theorem for coherent one-step conditional risk measures.
\begin{theorem}\label{rep_thm_coherent}
$\risk_{k}:\cs_{k+1}\rightarrow\cs_k$ is a coherent one-step conditional risk measure if and only if
\begin{align}
\risk_{k}(Z(x_{k+1}))=\sup_{\xi\in\csd_{k+1}(x_k,Q(x_{k+1}|x_k,u_k))}\sum_{x^\prime\in S} \xi(x^\prime)Z(x^\prime)\label{dual_rep}
\end{align}
where
\[
\begin{split}
&\csd_{k+1}(x_k,Q)=\\
&\left\{\xi\in\mathcal{M}\mid\xi\ll Q,\,\sum_{x^\prime\in S} \xi(x^\prime)Z(x^\prime)\leq\risk_{k}(Z),\, \forall Z\in\cs_{k+1}\right\},
\end{split}
\]
and
\[
\mathcal{M}=\left\{\xi\in \reals^{|S|}|\sum_{x^\prime\in S}\xi(x^\prime)=1, \,\,\xi(x^\prime)\geq 0,\,\forall x^\prime\in S \right\}.
\]
For $Z\in\cs_k\subseteq L_p(\Omega,\fil,P)$, $\csd_{k+1}(x_k,Q)$ is a subset of $L_q(\Omega,\fil,P)$, where $L_q(\Omega,\fil,P)$ is the dual space of $L_p(\Omega,\fil,P)$, for $1/p+1/q=1$ and $p,q\in[1,\infty]$.
\end{theorem}
\begin{proof}
Refer to Theorem 6.4 in \cite{Shapiro_Dentcheva_Ruszczynski_09} and references therein.
\end{proof}
The result essentially says that any coherent risk measure can be interpreted as an expectation taken with
respect to a worst-case measure, which is chosen from a suitable set of test measures \cite{Iancu_11}.

Furthermore, by Moreau-Rockafellar Theorem  (Theorem 7.4 in \cite{Shapiro_Dentcheva_Ruszczynski_09}), it implies $\csd_{k+1}(x_k,Q(x_{k+1}|x_k,u_k))=\partial\risk_k(0)$, when the transition probability kernel is $Q(x_{k+1}|x_k,u_k)$. The next Theorem implies a basic duality result on coherent risk measures.
\begin{theorem}\label{rep_thm_2}
$\risk_{k}:\cs_{k+1}\rightarrow\cs_k$ is a coherent one-step conditional risk measure, if and only if there exists a bounded, non-empty, weakly$^\ast$ compact and convex set: $\csd_{k+1}(x_k,Q(x_{k+1}|x_k,u_k))$ such that equation (\ref{dual_rep}) holds. Furthermore, if $\risk_k$ is a coherent risk measure, then it is continuous and sub-differentiable in $\cs_{k+1}$, also if $\text{dom}(\risk_k)=\{Z\in\cs_{k+1}:\risk_k(Z)<\infty\}$ has an non-empty interior, then $\risk_k $ is finite valued.
\end{theorem}
\begin{proof}
See Proposition 6.5, Theorem 6.6 and Theorem 6.7 in \cite{Shapiro_Dentcheva_Ruszczynski_09}.
\end{proof}
Since the analysis of this paper is restricted to finite state and action spaces, from this theorem, $\csd_{k+1}(x_k,Q(x_{k+1}|x_k,u_k))=\partial\risk_k(0)$ is a non-empty, convex, bounded and compact set in $\reals^{|S|}$. By extreme value theorem, the supremum in equation (\ref{dual_rep}) is attained.

\section{Discretization of the continuous risk thresholds in constrained dynamic programming}\label{sec:dis}
In the previous section, we have shown that the constrained stochastic optimal control problem can be solved using value iteration (See Theorem \ref{TC_good}). However, the constant risk threshold $r_k$ in value function $V_{k}(x_k,r_k)$, $k\in\{0,\ldots,N-1\}$ is a continuous state. This results in numerical complexity when value iteration is performed. Therefore, in this section, we consider a numerical approximation algorithm using discretization. First of all, from the dynamic programming operator possess several nice properties:
\begin{lemma}\label{prop_DP_op}
Let $V,\tilde{V}\in B(S \times \reals)$ be real-valued bounded functions and $T_k[V] : B(S \times \reals) \mapsto B(S \times \reals)$ be a dynamic programming operator in given in $B(S \times \reals)$ whose expression is given by equation (\ref{eq:T}) for any $k\in\{0,\ldots,N-1\}$. Then, the following statement holds:
\begin{enumerate}
\item \textbf{Monotonicity:} For any $(x,r)\in B(S\times\reals)$, if $V\leq \tilde V$, then $T_k[V](x,r)\leq T_k[\tilde{V}](x,r)$.
\item \textbf{Constant shift:} For any real number $L$ and $(x,r)\in B(S\times\reals)$, $T_k[V+L](x,r)=T_k[V](x,r)+L$, where $(V+L)(x,r):=V(x,r)+L$, $\forall\, (x,r)\in B(S\times\reals)$.
\item \textbf{Non-expansivity:} For all $V,\tilde{V}\in B(S \times \reals)$, $\|T_k[V]-T_k[\tilde{V}]\|_{\infty}\leq \|V-\tilde{V}\|_\infty$,
where $\|\cdot\|_\infty$ is the infinity norm of a function.
\end{enumerate}
\end{lemma}

Next, we introduce the method for constrained dynamic programming with discretized risk thresholds and updates.
\subsection{Dynamic programming with discretized risk thresholds and updates}
For $k\in\{0,\ldots,N-1\}$, we will partition $\Phi_k(x_k)$ into $t+1$ partitions using $t$ grid points: $\{\hat{r}_k^{(1)},\ldots,\hat{r}_k^{(t)}\}$ for every fixed $x_k\in S$. The step size of discretization of the risk thresholds $r_k$ is $\Delta$. For $\tau\in\{0,\ldots,t\}$, define the discretized region $\Phi^{(\tau)}_k(x_k)=[r_k^{(\tau)},r_k^{(\tau+1)})$, where $r_k^{(0)}=\underline R_N(x_k)$ and $r_k^{(t+1)}=\overline{R}_{N,k}+\epsilon$, for arbitrarily small $\epsilon>0$. We also define $\overline{\Phi}_k(x_k)=\{{r}_k^{(0)},\ldots,{r}_k^{(t+1)}\}$ to be a finite state of risk threshold at step $k$. Let $\tau\in\{0,\ldots,t\}$ such that $r_k\in\Phi^{(\tau)}_k(x_k)$. Now, define the approximation operator $T^{D}_{\Delta,k}$ for $x_k\in S$, $r_k\in \Phi^{(\tau)}_k(x_k)$:
\begin{equation}
T^{D}_{\Delta,k}[V](x_k, r_k) := \bar{T}^{D}_{\Delta,k}[V](x_k, r^{(\tau)}_k)\label{dis_DP_1}
\end{equation}
where
\begin{equation}
\begin{split}
\bar{T}^{D}_{\Delta,k}[V](x_k, r_k):=&\min_{(u,r^{D,\prime})\in F^D_k\left(x_k, r_k\right)} \, \biggl\{ c(x_k,u) \\ &+\sum_{x^{\prime}  \in S} \, Q(x^{\prime}|x_k,u) \,V(x^{\prime}, r^{D,\prime}(x^{\prime})) \biggr\},\label{dis_DP_2}
\end{split}
\end{equation}
where $F^D_k$ is the set of control/threshold \emph{functions}:
\begin{equation*}
\begin{split}
F^D_k(x_k,&r_k):= \bigg\{(u, r^{\prime}) \bigg |
u\in U(x_k), \,\, r^{D,\prime}(x^{\prime}) \in \overline{\Phi}_{k+1}(x^\prime), \\
&\forall x^{\prime} \in S, \,\, d(x_k, u)+\risk_k( r^{D,\prime}(x_{k+1})) \leq r_k
\bigg\}.
\end{split}
\end{equation*}
If $F^D_k(x_k,r_k) = \emptyset$, then $\bar{T}^{D}_{\Delta,k}[V_{k+1}](x_k, r_k)=\infty$.

By construction, we can see the set of optimal solution of $T^{D}_{\Delta,k}[V](x_k, r_k)$ is a subset of feasible space for the problem described by $T_k[V](x_k, r_k)$ (since $F^D_k(x_k,r_k)\subseteq F_k(x_k,r_k)$ and $r^{(\tau)}_k\leq r_k$). Because the solution of $T^{D}_{\Delta,k}[V](x_k, r_k)$ is an infimum over a finite set, the problem in (\ref{dis_DP_1}) is a minimization. Also, based on similar proofs, the dynamic programming operator $T^{D}_{\Delta,k}[V]$ satisfies all the properties given in Lemma \ref{prop_DP_op}. The main result of this section is to obtain a bound of the differences between  $T_k[V](x_k, r_k)$ and $T^{D}_{\Delta,k}[V](x_k, r_k)$, which will be given in the next subsection.

\subsection{Error bound analysis }
First, we have the following assumptions for the following analysis:
\begin{quote} {\textbf{Assumptions for discretization analysis:}}
\begin{enumerate}
\item \label{assume_lip} There exists $M_{c},M_{d}>0$ such that
\[
|c(x,u)-c({x},\tilde{u})|\leq M_{c}|u-\tilde{u}|,
\]
\[
|d(x,u)-d({x},\tilde{u})|\leq M_{d}|u-\tilde{u}|,
\]
for any $x\in S$, $u,\tilde{u}\in U(x)$.
\item \label{assume_Q} For any $u,\tilde{u}\in U(x_k)$, there exists $M_q>0$ such that \label{assume_Q}
\[
\sum_{x^\prime\in S}\left|Q(x^\prime|x_k,u)-Q(x^\prime|x_k,\tilde{u})\right|\leq M_{q}|u-\tilde{u}|.
\]
\end{enumerate}
\end{quote}
Assumptions (\ref{assume_lip}) to  (\ref{assume_Q}) are the critical assumptions required to perform error bound analysis in this section. First, we have following Proposition showing the Lipschsitz-ness  of set-valued mapping $\mathcal{U}_{k+1}(x_k,Q)$.
\begin{proposition}\label{prop_lip_xi}
For any $\xi\in\mathcal{U}_{k+1}(x_k,Q)$, there exists a ${M}_\xi>0$ such that for some $\tilde{\xi}\in\mathcal{U}_{k+1}(x_k,\tilde{Q})$,
\[
\sum_{x^\prime\in S}|\xi(x^\prime)-\tilde{\xi}(x^\prime)|\leq M_{\xi}\sum_{x^\prime\in S}\left|Q(x^\prime)-\tilde{Q}(x^\prime)\right|.
\]
\end{proposition}
\begin{proof}
From Theorem \ref{rep_thm_2}, we know that $\mathcal{U}_{k+1}(x_k,{Q})$ is a closed, bounded,  convex set of probability mass functions. Since any conditional probability mass function $Q$ is in the interior of $\text{dom}(\mathcal{U}_{k+1})$ and the graph of $\mathcal{U}_{k+1}(x_k,{Q})$ is closed, by Theorem 2.7 in \cite{nam_13}, $\mathcal{U}_{k+1}(x_k,{Q})$ is a Lipschitz set-valued mapping with respect to the Hausdorff distance.
Thus, for any $\xi\in\mathcal{U}_{k+1}(x_k,{Q})$, the following expression holds for some ${M}_\xi>0$:
\begin{equation*}
\inf_{\hat{\xi}\in\mathcal{U}_{k+1}(x_k,\tilde{Q})}\sum_{x^\prime\in S}|\xi(x^\prime)-\hat{\xi}(x^\prime)|\leq {M}_\xi\sum_{x^\prime\in S}\left|Q(x^\prime)-\tilde{Q}(x^\prime)\right|.
\end{equation*}
Next, we want to show that the infimum of the left side is attained. Since the objective function is convex, and $\mathcal{U}_{k+1}(x_k,\tilde{Q})$ is a convex compact set, there exists $\tilde{\xi}\in\mathcal{U}_{k+1}(x_k,\tilde{Q})$ such that infimum is attained.
\end{proof}
Next, we provide a Lemma that characterizes an upper bound for the magnitude of the value functions.
\begin{lemma}\label{V_bdd}
For $k\in\{0,\ldots,N-1\}$, the following bound is given for the value function $V_k(x_k,r_k)$: $\|V_{k}\|_\infty\leq (N-k)c_{\max}$, where
\begin{alignat}{1}
&c_{\text{max}}:=\max_{(x,u)\in S\times U} \, |c(x,u))|.\label{c_max}
\end{alignat}
\end{lemma}
\begin{proof}
First, from the definition of $V_{N}(x_N,r_N)$, we know that $V_{N}(x_N,r_N)=0$ for any $x_N\in S$, $r_N\in\Phi_N(x_N)$. Therefore, the above inequality holds for the for $k=N$. For $j\in\{0,\ldots,N-1\}$, since $|c(x_j,u_j)|\leq c_{\text{max}}$ for any $x_j\in S$, $u_j\in U(x_j)$, it implies $\|T_j[V_{N}]\|_\infty\leq  c_{\text{max}}$. Furthermore,
\begin{alignat*}{1}
\|V_{j}\|_\infty=&\|V_{j}-V_{N}\|_\infty\\
\leq&\|T_j[V_{j+1}]-T_j[V_{N}]\|_\infty+\|T_j[V_{N}]-V_{N}\|_{\infty}\\
\leq &\|V_{j+1}-V_{N}\|_\infty+c_{\text{max}}=\|V_{j+1}\|_\infty+c_{\text{max}}.
\end{alignat*}
The first inequality is due to triangle inequality and Theorem \ref{TC_good}, the second inequality is due to the non-expansivity property in Lemma \ref{prop_DP_op} and both equalities in the above expression are due to $V_{N}(x,r)=0$. Thus by recursion, we get
\[
\|V_{k}\|_\infty=\sum_{j=k}^{N-1}(\|V_{j}\|_\infty-\|V_{j+1}\|_\infty).
\]
and the proof is completed by noting that $\|V_{j}\|_\infty-\|V_{j+1}\|_\infty)\leq c_{\text{max}}$ for $j\in\{k,\ldots,N-1\}$.
\end{proof}
To prove the main result, we need the following technical Lemma.
\begin{lemma}\label{lem_u_r}
For every given $x_k\in S$ and $\tilde{r}_k,r_k\in\Phi_k(x_k)$, suppose Assumptions (\ref{assume_lip}) to (\ref{assume_Q}) hold. Also, define
 $\underline{r}^\prime:=\{{r}^\prime(x^\prime)\}_{x^\prime\in S}\in\reals^{|S|}$ and $\underline{\hat{r}}^\prime:=\{\hat{r}^\prime(x^\prime)\}_{x^\prime\in S}\in\reals^{|S|}$. If $F_k(x_k,r_k)$ and $F_k(x_k,\tilde{r}_k)$ are non-empty sets, then for any $(u,\underline{r}^\prime)\in F_k(x_k,r_k)$, there exists $(\hat{u},\underline{\hat{r}}^\prime)\in F_k(x_k,\tilde{r}_k)$ such that for some $M_{r,k}>0$,
\begin{alignat}{1}
& |u-\hat{u}|+\sum_{x^\prime\in S}|r^\prime(x^\prime)-\hat{r}^\prime(x^\prime)|\leq M_{r,k}|r_k-\tilde{r}_k|.\label{lip_cond_r}
\end{alignat}
\end{lemma}
\begin{proof}
First, we want to show that $\alpha(u,\underline{r}^\prime):=d(x_k,u)+\risk_k( r^{\prime}(x_{k+1}))$ is a Lipschitz function. Define $\{\xi^\ast(x^\prime)\}_{x^\prime\in S}\in\text{arg max}_{\xi\in\mathcal{U}_{k+1}(x_k,Q(x_{k+1}|x_k,{u}))}\bigg\{d(x_k,{u})+\sum_{x^\prime\in S} \xi(x^\prime)(r^{\prime}(x^\prime))\bigg\}$. Then, there exists a $\tilde{\xi}\in\mathcal{U}_{k+1}(x_k,Q(x_{k+1}|x_k,\tilde{u}))$ such that the following expressions hold:
\begin{equation}\label{alpha_diff}
\begin{split}
&\alpha(u,\underline{r}^\prime)-\alpha(\tilde{u},\underline{\tilde{r}}^\prime)\\
=&d(x_k,u)+\risk_k(r^{\prime}(x_{k+1}))-d(x_k,\tilde{u})-\risk_k(\tilde{r}^{\prime}(x_{k+1}))\\
\leq & d(x_k,u)-d(x_k,\tilde{u})+ \sum_{x^\prime\in S} ({\xi}^\ast(x^\prime)-\tilde{\xi}(x^\prime))r^{\prime}(x^\prime)\\
&+\sum_{x^\prime\in S}\tilde{\xi}(x^\prime)(r^{\prime}(x^\prime)-\tilde{r}^{\prime}(x^\prime))\\
\leq&|d(x_k,u)-d(x_k,\tilde{u})|+\sum_{x^\prime\in S}|r^\prime(x^\prime)-\tilde{r}^\prime(x^\prime)|\\
&+\max_{x\in S}|r^{\prime}(x)|\sum_{x^\prime\in S} |{\xi}^\ast(x^\prime)-\tilde{\xi}(x^\prime)|.
\end{split}
\end{equation}
The first equality follows from definitions of coherent risk measures. The first inequality is due to Theorem \ref{rep_thm_coherent} and the definition of ${\xi}^\ast\in\mathcal{U}_{k+1}(x_k,Q(x_{k+1}|x_k,u))$. The second inequality is due to the fact that $\tilde\xi$ is a probability mass functions in $\mathcal{U}_{k+1}(x_k,Q(x_{k+1}|x_k,\tilde{u}))$. Then, by Proposition \ref{prop_lip_xi}, there exists $M_\xi>0$ such that
\[
\sum_{x^\prime\in S}|\xi^\ast(x^\prime)-\tilde{\xi}(x^\prime)|\leq {M}_\xi\sum_{x^\prime\in S}\left|Q(x^\prime|x_k,u)-Q(x^\prime|x_k,\tilde{u})\right|.
\]
Furthermore, by Assumptions (\ref{assume_lip}) to (\ref{assume_Q}) and the definition of $\Phi_{k+1}(x_{k+1})$, expression (\ref{alpha_diff}) implies
\[
\alpha(u,\underline{r}^\prime)-\alpha(\tilde{u},\underline{\tilde{r}}^\prime)\leq M_{A,k}\left(|\tilde{u}-{u}|+\sum_{x^\prime\in S}|\tilde{r}^\prime(x^\prime)-{r}^\prime(x^\prime)|\right)
\]
where
\[
M_{A,k}=\max\left\{M_{d}+M_q(\overline{R}_{N,k})M_\xi,1\right\}.\\
\]
 By a symmetric argument, we can also show that
 \[
 \alpha(\tilde{u},\underline{\tilde{r}}^\prime)-\alpha(u,\underline{{r}}^\prime)\leq M_{A,k}\left(|\tilde{u}-{u}|+\sum_{x^\prime\in S}|\tilde{r}^\prime(x^\prime)-{r}^\prime(x^\prime)|\right).
 \]
 Thus, by combining both arguments, we have shown that $\alpha(u,\underline{r}^\prime)$ is a Lipschitz function. Next, for any $(u,r^\prime)\in F_k(x_k,r_k)$, where
 \[
 \begin{split}
 F_k(x_k,r)=\bigg\{&(u,r^\prime)|\,\,u\in U(x_k), \,\,r^\prime(x^\prime)\in\Phi_{k+1}(x^\prime),\\
 &\forall x^\prime\in S,\,\, \alpha(u,\underline{r}^\prime)\leq r\bigg\},
 \end{split}
 \]
 consider the following optimization problem:
\begin{alignat*}{1}
P_{x_k,u,\underline{r}^\prime}(r)=\inf_{(\tilde{u},\underline{\tilde{r}}^\prime)\in F_k(x_k,r)} & \,\,|\tilde{u}-u|+\sum_{x^\prime\in S}|\tilde{r}^\prime(x^\prime)-{r}^\prime(x^\prime)|.
\end{alignat*}
Since $(u,\underline{r}^\prime)$ is a feasible point of $F_k(x_k,r_k)$, $P_{x_k,u,\underline{r}^\prime}(r_k)=0$. By our assumptions, both $U(x_k)$ and $\Phi_{k+1}(x_{k+1})$ are compact sets of real numbers. Note that both $|\tilde{u}-u|+\sum_{x^\prime\in S}|\tilde{r}^\prime(x^\prime)-{r}^\prime(x^\prime)|$ and $\alpha(\tilde{u},\underline{\tilde{r}}^\prime)$ are Lipschitz functions in $(\tilde{u},\underline{\tilde{r}}^\prime)$. Also, consider the sub-gradient of $f(\tilde{u},\underline{\tilde{r}}^\prime,r):=\alpha(\tilde{u},\underline{\tilde{r}}^\prime)-{r}$\footnote{ A sub-gradient of a function $f:\mathcal{X}\rightarrow\reals$ at a point $x_0\in\mathcal{X}$ is a real vector $g$ such that for all $x\in\mathcal{X}$, $f(x)-f(x_0)\geq g^T(x-x_0)$, $\forall x\in\mathcal{X}$.}:
\[
\begin{split}
\partial f(\tilde{u},&\,\underline{\tilde{r}}^\prime,r)=\bigcap_{(\hat{u},\underline{\hat{r}}^\prime,\hat{r})\in\text{dom}(f)}\left\{\begin{bmatrix}
g_1\\
g_2\\
g_3
\end{bmatrix}\in\reals^{|U|}\times\reals^{|S|}\times\overline\reals\right.:\\
&\left. f(\hat{u},\underline{\hat{r}}^\prime,\hat{r})\geq f(\tilde{u},\underline{\tilde{r}}^\prime,r)+\begin{bmatrix}
g_1\\
g_2\\
g_3
\end{bmatrix}^T\left(\begin{bmatrix}
\tilde{u}\\
\underline{\tilde{r}}^\prime\\
r
\end{bmatrix}-\begin{bmatrix}
\hat{u}\\
\underline{\hat{r}}^\prime\\
\hat{r}
\end{bmatrix}\right)\right\}.
\end{split}
\]
\iftoggle{paper}
{
For any $(g_1,g_2,g_3)\in\partial f(u,\underline{{r}}^\prime,r)$, this implies
\[
\alpha(\tilde{u},\underline{\tilde{r}}^\prime)-\alpha(\hat{u},\underline{\hat{r}}^\prime)\geq(g_3+1)(r_k-\hat{r}_k)+g_1(\tilde{u}-\hat{u})+g_2^T(\underline{\tilde{r}}^\prime-\underline{\hat{r}}^\prime),
\]
$\forall (\hat{u},\underline{\hat{r}}^\prime,\hat{r})\in\text{dom} f$. Suppose $g_3>-1$, then there exists $\epsilon>0$ such that $g_3+1=\epsilon$. Also, by the Lipschitz-ness of $\alpha(\tilde{u},\underline{\tilde{r}}^\prime)$ and Cauchy Schwarz inequality, we get
\[
\begin{split}
&(M_{A,k}+|g_1|)|\tilde{u}-\hat{u}|+(1+\|g_2\|)\sum_{x^\prime\in S}|\tilde{r}^\prime(x^\prime)-\hat{r}^\prime(x^\prime)|\\
\geq&\epsilon(r-\hat{r}),\quad \forall\,\, (\hat{u},\underline{\hat{r}}^\prime,\hat{r})\in\text{dom} (f)
\end{split}
\]
Since $\tilde{u},\hat{u}$ are finite and $\underline{\tilde{r}}^\prime, \underline{\hat{r}}^\prime$ are bounded, the above inequality fails if $\hat{r}\rightarrow-\infty$. This yields a contradiction. Similarly, by considering $\hat{r}\rightarrow\infty$, we can also arrive at a contradiction for the case of $g_3<-1$. Therefore, the set of the third element of $\partial f(\tilde{u},\underline{\tilde{r}}^\prime,r)$ is a singleton and it equals to $\{-1\}$.
}
{
Since $\alpha(\tilde{u},\underline{\tilde{r}}^\prime)-{r}$ is differentiable on $r$, the third element of $\partial f(\tilde{u},\underline{\tilde{r}}^\prime,r)$ is a singleton and it equals to $\{-1\}$.
}
Next, consider the sub-gradient of $h(\tilde{u},\underline{\tilde{r}}^\prime,r)=|\tilde{u}-u|+\sum_{x^\prime\in S}|r^\prime(x^\prime)-\tilde{r}^\prime(x^\prime)|$. By identical arguments, we can show that the set of the third element of $\partial h(\tilde{u},\underline{\tilde{r}}^\prime,r)$ is a singleton and it equals to $\{0\}$. Therefore, Theorem 4.2 in \cite{Lucet_Ye_01} implies $P_{x_k,u,\underline{r}^\prime}(r)$ is strictly differentiable (Lipschitz continuous) in $r$ \footnote{ Theorem 4.2 in \cite{Lucet_Ye_01} implies both $\partial P_{x_k,u,\underline{r}^\prime}(r), \partial^\infty P_{x_k,u,\underline{r}^\prime}(r)\subseteq\{0\}$ for $r_k\in\Phi_k(r_k)$. This result further implies $P_{x_k,u,\underline{r}^\prime}(r)$ is strictly differentiable. For details, please refer to this paper.}. Then, for any $(u,\underline{r}^\prime)\in F_k(x_k,r_k)$, there exists $M_{r,k}>0$ such that
\[
\inf_{(\tilde{u},\underline{\tilde{r}}^\prime)\in F_k(x_k,\tilde{r}_k)}|\tilde{u}-u|+\sum_{x^\prime\in S}|r^\prime(x^\prime)-\tilde{r}^\prime(x^\prime)|\leq M_{r,k}|\tilde{r}_k-r_k|.
\]

Finally we want to show that the infimum on the left side of the above expression is attained. First, $|\tilde{u}-u|+\sum_{x^\prime\in S}|r^\prime(x^\prime)-\tilde{r}^\prime(x^\prime)|$ is coercive and continuous in $(\tilde{u},\underline{\tilde{r}}^\prime)$. By Example 14.29 in \cite{Rockafellar_98}, this function is a Caratheodory integrand and is also a normal integrand. Furthermore, since $F_k(x_k,\tilde{r}_k)$ is a closed set  (since $U(x_k)$ is a finite set, $\Phi_{k+1}(x_{k+1})$ is a compact set and the constraint inequality is non-strict)and $\alpha(\tilde{u},\underline{\tilde{r}}^\prime)-\tilde{r}_k$ is a normal integrand (see the proof of Theorem IV.2 in \cite{Chow_Pavone_13_1}), by Theorem 14.36 and Example 14.32 in \cite{Rockafellar_98}, one can show that the following indicator function:
\[
\mathbb{I}_{x_k}(\tilde{u},\underline{\tilde{r}}^\prime,\tilde{r}_k):=\left\{\begin{array}{ll}
0&\text{if $(\tilde{u},\underline{\tilde{r}}^\prime)\in F_k(x_k,\tilde{r}_k)$}\\
\infty&\text{otherwise}
\end{array}\right.
\]
is a normal integrand.  Furthermore, By Proposition 14.44 in \cite{Rockafellar_98}, the function
\[
g_{x_k}(\tilde{u},\underline{\tilde{r}}^\prime,\tilde{r}_k):=|\tilde{u}-u|+\sum_{x^\prime\in S}|r^\prime(x^\prime)-\tilde{r}^\prime(x^\prime)|+\mathbb{I}_{x_k}(\tilde{u},\underline{\tilde{r}}^\prime,\tilde{r}_k)
\]
is a normal integrand. Also, $\inf_{\tilde{u}}g_{x_k}(\tilde{u},\underline{\tilde{r}}^\prime,\tilde{r}_k)=\inf_{(\tilde{u},\underline{\tilde{r}}^\prime)\in F_k(x_k,\tilde{r}_k)}|\tilde{u}-u|+\sum_{x^\prime\in S}|r^\prime(x^\prime)-\tilde{r}^\prime(x^\prime)|$. By Theorem 14.37 in \cite{Rockafellar_98}, there exists $(\hat{u},\underline{\hat{r}}^\prime)\in F_k(x_k,\tilde{r}_k)$ such that $(\hat{u},\underline{\hat{r}}^\prime)\argmin g_{x_k}(\tilde{u},\underline{\tilde{r}}^\prime,\tilde{r}_k)$. Furthermore, the right side of the above equality is finite since $F_k(x_k,\tilde{r}_k)$ is a non-empty set. The definition of $\mathbb{I}_{x_k}(\hat{u},\underline{\hat{r}}^\prime,\tilde{r}_k)$ implies that $(\hat{u},\underline{\hat{r}}^\prime)\in F_k(x_k,\tilde{r}_k)$. Therefore this implies expression (\ref{lip_cond_r}) holds for any $(u,\underline{r}^\prime)\in F_k(x_k,r_k)$.
\end{proof}
The following Lemma provides a sensitivity condition for the value function $V_{k}(x_k,r_k)$.
\begin{lemma}\label{V_sens}
Suppose $F_k(x_k,r_k)$ and $F_k(x_k,\tilde{r}_k)$ are non-empty sets for $k\in\{0,\ldots,N-1\}$. Then, for $x_k\in S$, $r_k,\tilde{r}_k\in\Phi_k(x_k)$, such that $r_k\geq\tilde{r}_k$, $k\in\{0,\ldots,N\}$, the following expression holds:
\begin{align}
& 0\leq V_k(x_k, \tilde{r}_k)-V_k(x_k, r_k)\leq M_{V,k}(r_k-\tilde{r}_k)\label{V_bdd_1}
\end{align}
where $M_{V,k}=(M_{c}+M_q(N-k-1)c_{\text{max}}+M_{V,k+1})M_{r,k}>0$, and $M_{V,N}=0$.
\end{lemma}
\begin{proof}
First, for $k\in\{0,\ldots,N-1\}$, when $\tilde{r}_k\leq r_k$, by Lemma IV.1 in \cite{Chow_Pavone_13_1}, we know that  $V_{k}(x_k,\tilde{r}_k)\geq V_{k}(x_k,r_k)$. The proof is completed if we can show that for $\tilde{r}_k\leq r_k$,
\[
V_k(x_k, \tilde{r}_k)-V_k(x_k, r_k)\leq M_{V,k}(r_k-\tilde{r}_k).
\]
First, at $k=N$, for any $r_N,\tilde{r}_N\in\Phi_N(x_N)$, we get $V_{N}(x_N,\tilde{r}_N)=V_{N}(x_N,r_N)=0$. Inequality (\ref{V_bdd_1}) trivially holds for any $M_{V,N}>0$. By induction's assumption, suppose there exists $M_{V,j+1}>0$ such that following inequality holds at $k=j+1$:
\[
\left|V_{j+1}(x,\tilde{r}_{j+1})-V_{j+1}(x,r_{j+1})\right|\leq M_{V,j+1}\left|\tilde{r}_{j+1}-{r}_{j+1}\right|.
\]
for any $x\in S$. Then, for the case at $k=j$, by Theorem IV.2 in \cite{Chow_Pavone_13_1}, the infimum of $T_j[V_{j+1}]$ is attained. From Theorem \ref{TC_good}, $V_j(x_j, r_j)=T_j[V_{j+1}](x_j, r_j)$. For any given $x_j\in S$, $r_j\in \Phi_j(x_j)$, let $(u^\ast_j,r^{\ast,\prime})$ be the minimizer of $T_j[V_{j+1}](x_j, r_j)$. Then, there exists $(\hat{u}_j,\hat{r}^{\prime})\in F_j(x_j,\tilde{r}_j)$, such that inequality (\ref{lip_cond_r}) and the following expressions hold:
\begin{alignat*}{1}
&V_{j}(x_j,\tilde{r}_j)-V_{j}(x_j,r_j)\\
\leq&c(x_j,\hat{u}_j)-c(x_j,u^\ast_j)+\sum_{x^{\prime}  \in S} \, Q(x^{\prime}|x_j,\hat{u}_j)V_{j+1}(x^{\prime}, \hat{r}^{\prime}(x^\prime))\\
&-\sum_{x^{\prime}  \in S} \, Q(x^{\prime}|x_j,{u}^\ast_j)V_{j+1}(x^{\prime}, r^{\ast,\prime}(x^{\prime}))\\
=&c(x_j,\hat{u}_j)-c(x_j,u^\ast_j)\\
&+  \sum_{x^{\prime}  \in S} \, Q(x^{\prime}|x_j,\hat{u}_j)\, \left(V_{j+1}(x^{\prime}, \hat{r}^{\prime}(x^{\prime}))- V_{j+1}(x^{\prime}, r^{\ast,\prime}(x^{\prime}))\right)\\
&+\sum_{x^{\prime}  \in S} \,  \left( Q(x^{\prime}|x_j,\hat{u}_j)- Q(x^{\prime}|x_j,u^\ast_j)\right)\, V_{j+1}(x^{\prime}, r^{\ast,\prime}(x^{\prime}))\\
\leq&\|V_{j+1}\|_\infty\sum_{x^{\prime}  \in S} \,  \left|Q(x^{\prime}|x_j,\hat{u}_j)- Q(x^{\prime}|x_j,u^\ast_j)\right|\\
&+\sum_{x^\prime\in S}\bigg\{\left|V_{j+1}(x^{\prime}, r^{\ast,\prime}(x^{\prime}))- V_{j+1}(x^{\prime}, \hat{r}^{\prime}(x^{\prime}))\right|\bigg\}\\
&+|c(x_j,\hat{u}_j)-c(x_j,u^\ast_j)|.
\end{alignat*}
The first inequality follows from the definitions. The second inequality follows from $\sum_{x^{\prime}  \in S} \, Q(x^{\prime}|x_j,\hat{u}_j)=1$ and the definition of $\|V_{j+1}\|_\infty$ and $c_{\text{max}}$. From Assumption (\ref{assume_lip}) and Inductions' assumption, the above expression further implies
\begin{alignat}{1}
&V_{j}(x_j,\tilde{r}_j)-V_{j}(x_j,r_j)\nonumber\\
\leq&(M_{c}+M_q\|V_{j+1}\|_\infty)|\hat{u}_j-u^\ast_j|\nonumber\\
&+M_{V,j+1}\sum_{x^\prime\in S}|\hat{r}^{\prime}(x^{\prime})-r^{\ast,\prime}(x^{\prime})|\nonumber\\
\leq &(M_{c}+M_q\|V_{j+1}\|_\infty+M_{V,j+1})M_{r,j}|\tilde{r}_j-r_j|.\label{V_non_rand}
\end{alignat}
The last inequality is simply resulted from by Lemma \ref{lem_u_r}. In addition, from Lemma \ref{V_bdd}, we get
\[
\begin{split}
\|V_{j+1}\|_\infty=&\sum_{i=j+1}^{N-1}\|V_i\|_\infty-\|V_{i+1}\|_\infty\leq (N-j-1)c_{\text{max}}.
\end{split}
\]
Then, by applying this inequality to the expression derived in the previous part of the proof, we get
\begin{alignat}{1}
&V_{j}(x_j,\tilde{r}_j)-V_{j}(x_j,r_j)\\
\leq&\left(M_{c}+M_q(N-j-1)c_{\text{max}}+M_{V,j+1}\right)M_{r,j}|\tilde{r}_j-r_j|.\nonumber
\end{alignat}
Thus by induction, expression (\ref{V_bdd_1}) holds.
\end{proof}
The next Lemma shows that the difference between dynamic programming operators $\bar{T}^{D}_{\Delta,k}[V_{k+1}](x_k, r_k)$ and ${T}_{k}[V_{k+1}](x_k, r_k)$ is bounded.
\begin{lemma}\label{DP_delta}
For any $x_k\in S$, $r_k\in\Phi_k(x_k)$, the following inequality holds for $k\in\{0,\ldots,N-1\}$:
\[
0\leq \bar{T}^{D}_{\Delta,k}[V_{k+1}](x_k, r_k)-{T}_{k}[V_{k+1}](x_k, r_k)\leq M_{V,k+1}\Delta
\]
where $M_{V,k+1}>0$ is given by Lemma \ref{V_sens} and $\Delta$ is the step size of the discretization of risk threshold $r_{k}$.
\end{lemma}
\begin{proof}
First, by the definition of $F^D_k(x_k,r_k)$, we know that $F^D_k(x_k,r_k)\subseteq F_k(x_k,r_k)$. Since, the objective functions and all other constraints in $\bar{T}^{D}_{\Delta,k}[V_{k+1}](x_k, r_k)$ and ${T}_{k}[V_{0,k+1}](x_k, r_k)$ are identical, we can easily conclude that $\bar{T}^{D}_{\Delta,k}[V_{k+1}](x_k, r_k)\geq {T}_{k}[V_{k+1}](x_k, r_k)$ for all $x_k\in S$, $r_k\in\Phi_k(x_k)$. The proof is completed if we can show
\[
\bar{T}^{D}_{\Delta,k}[V_{k+1}](x_k, r_k)-{T}_{k}[V_{k+1}](x_k, r_k)\leq M_{V,k+1}\Delta.
\]
By Theorem IV.2 in \cite{Chow_Pavone_13_1} we know that the infimum of ${T}_{k}[V_{k+1}](x_k, r_k)$ is attained. Let $(u_k^\ast,r^{\ast,\prime})\in F_k(x_k,r_k)$ be the minimizer of $T_k[V_{k+1}](x_k, r_k)$. Also, for every fixed $x^\prime\in S$, let $\tau(x^\prime)\in\{0,\ldots,t\}$ such that $r^{\ast,\prime}(x^\prime)\in \Phi^{(\tau(x^\prime))}_{k+1}(x^\prime)$. Now, construct
\[
\tilde{r}^{\prime}(x^\prime):=r_{k+1}^{(\tau(x^\prime))}\in\Phi^{(\tau(x^\prime))}_{k+1}(x^\prime).
\]
By definition of $\overline{\Phi}_{k+1}(x^\prime)$, we know that $\tilde{r}^{\prime}(x^\prime)\in\overline{\Phi}_{k+1}(x^\prime)$, $\forall x^\prime\in S$. Since $r_{k+1}^{(\tau(x^\prime))}$ is the lower bound of $\Phi^{(\tau(x^\prime))}_{k+1}(x^\prime)$, we have $r_{k+1}^{(\tau(x^\prime))}\leq r^{\ast,\prime}(x^\prime)$. Furthermore, since the size of $\Phi^{(\tau(x^\prime))}_{k+1}(x^\prime)$ is $\Delta$, we know that $|r_{k+1}^{(\tau(x^\prime))}-r^{\ast,\prime}(x^\prime)|\leq\Delta$ for any $x^\prime\in S$. By monotonicity of coherent risk measures,
\[
\begin{split}
&d(x_k,u_k^\ast)+\risk_k(\tilde{r}^{\prime}(x_{k+1}))\leq d(x_k,u_k^\ast)+ \risk_k(r^{\ast,\prime}(x_{k+1}))\leq r_k.
\end{split}
\]
Therefore, we conclude that $(u_k^\ast,\tilde{r}^{\prime})\in F^D_k(x_k,r_k)$ is a feasible solution to the problem in $\bar{T}^{D}_{\Delta,k}[V_{k+1}](x_k, r_k)$. From this fact, we get the following inequalities:
\[
\begin{split}
& \bar{T}^{D}_{\Delta,k}[V_{k+1}](x_k, r_k)-{T}_{k}[V_{k+1}](x_k, r_k)\\
\leq&\sum_{x^{\prime}  \in S} \, Q(x^{\prime}|x_k,{u}^\ast_k)\bigg(V_{k+1}(x^{\prime}, \tilde{r}^{\prime}(x^{\prime}))-V_{k+1}(x^{\prime}, {r}^{\ast,\prime}(x^{\prime}))\bigg)\\
\leq& \sup_{x^\prime\in S}\bigg\{\left|V_{k+1}(x^{\prime}, \tilde{r}^{\prime}(x^{\prime}))- V_{k+1}(x^{\prime}, r^{\ast,\prime}(x^{\prime}))\right|\bigg\}\\
\leq& M_{V,k+1}\sup_{x^\prime\in S}|\tilde{r}^{\prime}(x^{\prime})-r^{\ast,\prime}(x^{\prime}))|\leq M_{V,k+1}\Delta.
\end{split}
\]
The first inequality is due to substitutions of the feasible solution of $\bar{T}^{D}_{\Delta,k}[V_{k+1}](x_k, r_k)$ and the optimal solution of ${T}_{k}[V_{k+1}](x_k, r_k)$. The second inequality is trivial. The third inequality is a result of Lemma \ref{V_sens} and the fourth inequality is due to the definition of $\tilde{r}^{\prime}(x^\prime)$, for all $x^\prime\in S$. This completes the proof.
\end{proof}
The following Lemma is the main result of this section. It characterizes the error bound between the dynamic programming operator $T_k[V_{k+1}](x_k,r_k)$ and $T^{D}_{\Delta,k}[V_{k+1}](x_k,r_k)$.
\begin{lemma}\label{thm_diff_DP}
Suppose Assumptions (\ref{assume_lip}) to (\ref{assume_Q}) hold. Then, there exists a constant $M_{V,k}>0$ such that
\begin{equation}
\begin{split}
\|T^{D}_{\Delta,k}[V_{k+1}]-T_k[V_{k+1}]\|_\infty\leq (M_{V,k}+M_{V,k+1})\Delta\label{exp_approx_bdd_1}
\end{split}
\end{equation}
where $T^{D}_{\Delta,k}[V_{k+1}](x, r)$ is defined in equation (\ref{dis_DP_1}), $\Delta$ is the step size of the discretization of risk threshold $r_k$ and the expression of $M_{V,k},M_{V,k+1}>0$ is given in Lemma \ref{V_sens}, for $k\in\{0,\ldots,N-1\}$.
\end{lemma}
\begin{proof}
For any given $x_k\in S$ and $r_k\in\Phi_k(x_k)$, let $\tau\in\{0,\ldots,t\}$ such that $r_k\in\Phi^{(\tau)}_k(x_k)$. Then, by the definition of $T^{D}_{\Delta,k}[V_{k+1}](x_k, r_k)$ and Theorem \ref{TC_good}, the following expression holds:
\[
\begin{split}
&|T^{D}_{\Delta,k}[V_{k+1}](x_k, r_k)-T_k[V_{k+1}](x_k, r_k)| \leq|V_{k}(x_k,r_k^{(\tau)})- \\
&V_{k}(x_k,r_k)|+|\bar{T}^{D}_{\Delta,k}[V_{k+1}](x_k, r^{(\tau)}_k)-T_k[V_{k+1}](x_k,r^{(\tau)}_k)|.
\end{split}
\]
Also, by using Lemma \ref{V_sens} and \ref{DP_delta}, the above expression implies that
\[
\begin{split}
&\left |T^{D}_{\Delta,k}[V_{k+1}](x_k, r_k)-T_k[V_{k+1}](x_k, r_k)\right |\\
\leq& M_{V,k+1}\Delta+M_{V,k}|r_k-r_k^{(\tau)}|\leq (M_{V,k}+M_{V,k+1})\Delta.
\end{split}
\]
The last inequality follows from the fact that $r_k\in\Phi^{(\tau)}_k(x_k)$ implies $|r_k^{(\tau)}-r_k|\leq \Delta$, where $r_k^{(\tau)}$ is the lower bound of the discretized region of risk threshold: $\Phi^{(\tau)}_k(x_k)$. By taking supremum of $x_k\in S$ and $r_k\in\Phi_k(x_k)$ on both sides of the resultant inequality, we conclude the inequality given in expression (\ref{exp_approx_bdd_1}).
\end{proof}
Next, define $M_r=\max_{k\in\{0,\ldots,N-1\}}M_{r,k}$. The following Theorem provides an error bound between the value function: $V_{k}(x_k,r_k)$ and the value function with discretizations: $V_{k}^{D}(x_k,r_k)$.
\begin{theorem}\label{lem_diff_val_itr}
Define $V_{k}^{D}(x_k,r_k):=T^{D}_{\Delta,k}[V^{D}_{k+1}](x_k,r_k)$, $k\in\{0,\ldots,N-1\}$ as the value function with discretized risk threshold/update where $V^{D}_{N}(x_{N}, r_{N}):=V_{N}(x_{N}, r_{N})=0$. Suppose Assumptions (\ref{assume_lip}) to (\ref{assume_Q}) hold. Then,
\begin{equation*}
\begin{split}
&\|V^{D}_{k}-V_{k}\|_\infty\leq \,2\Delta\bigg(\frac{(M_rM_qc_{\text{max}}-M_c(1-M_r))(1-M_r^N)}{(1-M_r)^3}\\
&+\frac{N(N-1)M_rM_q c_{\text{max}}}{2(1-M_r)}+\frac{N(M_c(1-M_r)-M_qM_rc_{\text{max}})}{(1-M_r)^2}\bigg)
\end{split}
\end{equation*}
where $\Delta$ is the step size of the of risk threshold discretization.
\end{theorem}
\begin{proof}
From Theorem \ref{thm_diff_DP}, we know that for $j\in\{k,\ldots,N-1\}$, $\|T^{D}_{\Delta,j}[V_{j+1}]-T_{j}[V_{j+1}]\|_\infty\leq (M_{V,j}+M_{V,j+1})\Delta$, where $\Delta$ is the step size of the discretization of risk threshold $r_{j}$. Therefore, we have the following expressions:
\begin{alignat*}{1}
 &\|V^{D}_{j}-V_{j}\|_\infty=\|T^{D}_{\Delta,j}[V^{D}_{j+1}]-T_{j}[V_{j+1}]\|_\infty\\
  \leq&  \|T^{D}_{\Delta,j}[V^{D}_{j+1}]-T^{D}_{\Delta,j}[V_{j+1}]\|_\infty+ \|T^{D}_{\Delta,j}[V_{j+1}]-T_{j}[V_{j+1}]\|_\infty\\
 \leq & \|V^{D}_{j+1}-V_{j+1}\|_\infty+(M_{V,j}+M_{V,j+1})\Delta.
\end{alignat*}
The first equality is due to Theorem \ref{TC_good} and the fact that $V_{j}^{D}(x_j,r_j)=T^{D}_{\Delta,j}[V^{D}_{j+1}](x_j,r_j)$. The third inequality is based on the non-expansivity property in Lemma \ref{prop_DP_op} and the arguments in Theorem \ref{thm_diff_DP}. Furthermore,
\[
\begin{split}
 \|V^{D}_{k}-V_{k}&\|_\infty=\sum_{j=k}^{N-1} \Bigl(\|V^D_{j}-V_{j}\|_\infty- \|V^D_{j+1}-V_{j+1}\|_\infty\Bigr)\\
 \leq&\left(\sum_{j=k}^{N-1}M_{V,j}+M_{V,j+1}\right)\Delta\leq 2\left(\sum_{j=0}^{N-1}M_{V,j}\right)\Delta.
\end{split}
\]
\iftoggle{paper}
{
Now, recall
\[
\begin{split}
M_{V,k}=&(M_{c}+M_q(N-k-1)c_{\text{max}}+M_{V,k+1})M_{r,k}\\
\leq&(M_{c}+M_q(N-k-1)c_{\text{max}}+M_{V,k+1})M_{r}
\end{split}
\]
By geometric series, we can show that
\[
M_{V,k}\leq \frac{kM_rM_q c_{\text{max}}+M_c(1-M_r^k)}{1-M_r}-\frac{M_rM_qc_{\text{max}}(1-M_r^k)}{(1-M_r)^2}
\]}
{}
Therefore, the proof is completed by summing the right side of the inequality from $0$ to $N-1$ and combining all previous arguments.
\end{proof}
As the step size $\Delta\rightarrow 0$, for any $x_k\in S$ and $r_k\in\Phi_k(x_k)$, this Theorem implies that $V_k^D(x_k,r_k)\rightarrow V_k(x_k,r_k)$.
\begin{remark}\label{curse_dim}
Unfortunately, similar to all multi-grid discretization approaches discussed in \cite{Chow_Tsitsiklis_91}, \cite{tsitsiklis_van_roy_96}, \cite{gordon_99}, the multi-grid discretization algorithm in this paper also suffers from the curse of dimensionality. Suppose the number of discretized grid  used is $|R|$. For each time horizon, the size of state space is $|S||R|$. However, the size of the action space is $|A|(|R|)^{|S|}$. Methods such as Branch and bound or rollout algorithms can be applied to find the minimizers in each step to alleviate this issue, if the upper/lower bounds of the value functions are effectively calculated.
\end{remark}

\section{Numerical Implementation}\label{sec:num}
Consider an example with $3$ states ($x\in\{1,2,3\}$), $2$ available actions ($u\in\{1,2\}$) with time horizon $N=3$. The costs, constraint costs and transition probabilities are given as follows:
\[
\begin{split}
&\small{\begin{bmatrix}
c(1,1)&c(1,2)\\
c(2,1)&c(2,2)\\
c(3,1)&c(3,2)\\
\end{bmatrix}=\begin{bmatrix}
1&3\\
2&4\\
5&6\\
\end{bmatrix},\,\begin{bmatrix}
d(1,1)&d(1,2)\\
d(2,1)&d(2,2)\\
d(3,1)&d(3,2)\\
\end{bmatrix}=\frac{1}{10}\begin{bmatrix}
5&4\\
6&3\\
5&1
\end{bmatrix}},\\
&\small{Q(x^\prime|x,1)=\begin{bmatrix}
0.2 &0.5 &0.3\\
0.4 &0.3& 0.3\\
 0.3 &0.3& 0.4
\end{bmatrix},\, Q(x^\prime|x,2)=\begin{bmatrix}
0.3& 0.5& 0.2\\
0.2& 0.3& 0.5\\
0.3& 0.4& 0.3\\
\end{bmatrix}}.
\end{split}
\]
For any $x_0\in S$ and $r_0\in\Phi_0(x_0)$, the risk sensitive constrained stochastic optimal control problem  we are solving is as follows:
\begin{alignat*}{2}
\min_{\pi \in \Pi} & & \quad&\expectation{\sum_{k=0}^{2}\, c(x_k, u_k)} \\
\text{subject to} & & \quad&\risk_{0,3}\Bigl(d(x_0,u_0), d(x_1,u_1), d(x_{2},u_{2}),0\Bigr) \leq r_0.
\end{alignat*}
where $u_k=\pi_k(h_{0,k})$ for $k\in\{0,1,2\}$, $\risk_{0,N}(Z_0,Z_1,Z_2,Z_3)= Z_0 + \risk_0(Z_{1} + \risk_{1}(Z_{2}+\risk_{2}(Z_3)))$ and
\begin{align*}
\risk_k(V) = \expectation{V} + 0.2\, \Bigl (\expectation{[V - \expectation{V}]_{+}^2} \Bigr)^{1/2}.
\end{align*}
First, this problem can be re-casted using multi-stage constrained dynamic programming using the methods described by Theorem IV.3 in \cite{Chow_Pavone_13_1}. Furthermore, based on equations (\ref{dis_DP_1}) to (\ref{dis_DP_2}), we can approximate the optimal value function using risk threshold/update discretization. In this example, we discretize every risk threshold sets into $M$ regions, where
\[
M\in\{5,10,20,40,60,80,100,150\}.
\]
With different sizes of risk threshold discretization, we get approximations of optimal value functions, up to various degrees of accuracies. Figure \ref{fig_st} shows both the approximations of value function using various step sizes and their errors of approximations. As the number of $M$ increases, the approximated value function converges towards the true optimal value function. However, as discussed in Remark \ref{curse_dim}, the size of action space increases exponentially with the number of states, thus it makes enumerating all state/action pairs during value iteration computationally expensive.

\begin{figure}[h]
\centering
{
  \includegraphics[width = 0.5\textwidth]{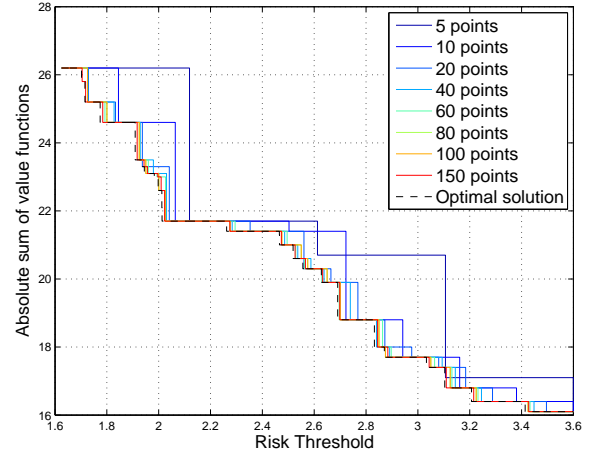}
      }
      \caption{Convergence of approximated value functions, and errors of approximations.}
      \label{fig_st}
\end{figure}

\section{Conclusion}\label{sec:conc}
In this paper we have presented and analyzed an uniform grid discretization algorithm for approximating the Bellman's recursion for finite horizon constrained stochastic optimal control problems. Although the current algorithm suffers from curse of dimensionality, it is by far the only known algorithm for numerically approximating constrained dynamic programming algorithms with continuous risk updates. This paper also leaves important extensions open for further researches that involve randomized grid sampling and variable resolution of discretization.

\bibliographystyle{unsrt}
\bibliography{ref_dyn_pro2}

\begin{thebibliography}{10}

\bibitem{Chow_Pavone_13_1}
Y.~Chow and M.~Pavone.
\newblock Stochastic optimal control with dynamic, time-consistent risk
  constraints.
\newblock In {\em American Control Conference}, 2013.

\bibitem{bertsekas_05}
D.~Bertsekas.
\newblock {\em Dynamic programming and optimal control}.
\newblock Athena Scientific, 2005.

\bibitem{Chen_07}
R.~Chen and E.~Feinberg.
\newblock Non-randomized policies for constrained {Markov} decision process.
\newblock {\em Mathematical Methods in Operations Research}, 66:165--179, 2007.

\bibitem{rus_09}
A.~Ruszczynski.
\newblock {Risk-averse dynamic programming for Markov decision process}.
\newblock {\em Journal of Mathematical Programming}, 125(2):235--261, 2010.

\bibitem{Shapiro_Dentcheva_Ruszczynski_09}
A.~Shapiro, D.~Dentcheva, and A.~Ruszczynski.
\newblock {\em Lectures on Stochastic Programming: Modeling and Theory}.
\newblock SIAM, 2009.

\bibitem{Iancu_11}
P.~Huang, D.~Iancu, M.~Petrik, and D.~Subramanian.
\newblock The price of dynamic inconsistency for distortion risk measures.
\newblock {\em arXiv preprint arXiv:1106.6102}, 2011.

\bibitem{whitt_78}
W.~Whitt.
\newblock {Approximations of dynamic programs, Volume I}.
\newblock {\em Mathematics of Operations Research}, 3(3):231--243, 1978.

\bibitem{gordon_99}
G.~Gordon.
\newblock {Approximate solutions to Markov decision processes}.
\newblock {\em Robotics Institute}, page 228, 1999.

\bibitem{Chow_Tsitsiklis_91}
C.~Chow and J.~Tsitsiklis.
\newblock An optimal one-way multigrid algorithm for discrete-time stochastic
  control.
\newblock {\em IEEE Transactions on Automatic Control}, 36(8):898--914, 1991.

\bibitem{Rust_97}
J.~Rust.
\newblock Using randomization to break the curse of dimensionality.
\newblock {\em Econometrica: Journal of the Econometric Society}, pages
  487--516, 1997.

\bibitem{tsitsiklis_van_roy_96}
J.~Tsitsiklis and B.~Van Roy.
\newblock An analysis of temporal-difference learning with function
  approximation.
\newblock {\em IEEE Transactions on Automatic Control}, 42(5):674--690, 1997.

\bibitem{Munos_Moore_00}
R.~Munos and A.~Moore.
\newblock Rates of convergence for variable resolution schemes in optimal
  control.
\newblock 2000.
\newblock Available at
  \url{http://citeseerx.ist.psu.edu/viewdoc/summary?doi=10.1.1.41.1488}.

\bibitem{Munos_Moore_98}
R.~Munos and A.~Moore.
\newblock Barycentric interpolators for continuous space and time reinforcement
  learning.
\newblock In {\em Neural Information Processing Systems}, 1998.

\bibitem{Munos_Moore_97}
R.~Munos.
\newblock A convergent reinforcement learning algorithm in the continuous case:
  the finite-element reinforcement learning.
\newblock 1996.
\newblock Available at
  \url{http://www.ri.cmu.edu/pub_files/pub1/munos_remi_1996_1/munos_remi_1996_1.pdf}.

\bibitem{nam_13}
N.~Nguyen and G.~Lafferriere.
\newblock Lipschitz properties of non-smooth functions and set-valued mappings
  via generalized differentiation and applications.
\newblock {\em arXiv preprint arXiv:1302.1794}, 2013.

\bibitem{Lucet_Ye_01}
Y~Lucet and J.~Ye.
\newblock Sensitivity analysis for the value function for optimization problems
  with variational inequalities constraints.
\newblock {\em SIAM J. OPTIM}, 40(3):699--723, 2002.

\bibitem{Rockafellar_98}
R.~Rockafellar and R.~Wets.
\newblock {\em Variational Analysis}.
\newblock Springer, 1998.

\end{thebibliography}

\addtolength{\textheight}{-14cm}   
\end{document}